\documentclass{article}

\usepackage[utf8]{inputenc}
\usepackage{authblk}
\usepackage{setspace}
\usepackage[margin=1.25in]{geometry}
\usepackage{graphicx}
\graphicspath{ {./figures/} }
\usepackage{subcaption}
\usepackage{amsmath}
\usepackage{lineno}
\usepackage{amsfonts}
\usepackage{amssymb, amsthm, amsmath, amsfonts}
\usepackage{wasysym}
\usepackage{mathrsfs}
\usepackage{hyperref}
\usepackage{url}
\usepackage{graphicx}
\usepackage{lineno}
\usepackage[colorinlistoftodos]{todonotes}
\usepackage{listings}
\usepackage{cancel, enumerate}
\usepackage{rotating, environ}
\usepackage{caption}
\usepackage{subcaption}
\usepackage[inline]{enumitem}
\usepackage{dirtree}
\usepackage{xcolor}
\usepackage{tipa}
\usepackage{float}
\usepackage{changepage}
\usepackage{tikz-cd}

\newtheorem{defn}{Definition}
\newtheorem{prop}{Proposition}
\newtheorem{lemma}{Lemma}

\newtheorem{ex}{Example}

\newtheorem{rmk}{Remark}
\newtheorem{qstn}{Question}
\newtheorem{thm}{Theorem}

\newtheorem{app}{Appendix}

\newcommand{\subseq}{\subseteq}
\newcommand{\N}{\mathbb{N}}
\newcommand{\Z}{\mathbb{Z}}

\newcommand{\R}{\mathbb{R}}

\newcommand{\ds}{\displaystyle}

\newcommand{\Aut}{\text{Aut}}

\newcommand{\F}{\mathbb{F}}

\renewcommand{\O}{\mathcal{O}}

\makeatletter
\newcommand{\tpitchfork}{%
  \vbox{
    \baselineskip\z@skip
    \lineskip-.52ex
    \lineskiplimit\maxdimen
    \m@th
    \ialign{##\crcr\hidewidth\smash{$-$}\hidewidth\crcr$\pitchfork$\crcr}
  }%
}
\makeatother

\renewcommand{\S}{\mathfrak{S}}

\renewcommand{\v}{\mathbf{v}}
\newcommand{\w}{\mathbf{w}}
\newcommand{\tr}{\text{tr}}
\renewcommand{\Re}{\mathbf{Re}}
\renewcommand{\Im}{\mathbf{Im}}
\newcommand{\OO}{\mathbb{O}}

\usepackage[style=nejm, 
citestyle=numeric-comp,
sorting=none]{biblatex}
\title{Oriented Steiner Triple Systems, Steiner Products, and Dynamics}

\author{Jake Kettinger and Chris Peterson}

\affil{Department of Mathematics, Colorado State University}

\date{}

\begin{document}

\maketitle

\begin{abstract}
Let $\mathfrak S$ denote a Steiner triple system on an $n$-element set. An orientation of $\mathfrak S$ is an assignment of a cyclic ordering to each of the triples in $\mathfrak S$. From an oriented Steiner triple system, one can define an anticommutative bilinear operation on $\mathbb R^n$ resembling the cross product. We call this bilinear operation a Steiner product. We classify the oriented Steiner triple systems on sets of size 7 and 9 and investigate the dynamics of their associated Steiner products.
\end{abstract}


\section{Introduction}
Steiner triple systems seem to be first introduced by Wesley Woolhouse in 1844 in a problem posed in {\it The Lady's and Gentlemen's Diary}. The problem posed by Woolhouse was solved by Thomas Kirkman in 1847 \cite{kirkman1847} and a refinement of the problem, known as {\it Kirkman's Schoolgirl Problem}, was proposed by Kirkman in 1850 (again in The Lady's and Gentlemen's Diary). In 1853, Jakob Steiner independently rediscovered the concept and described these systems with a more precise mathematical language. Steiner systems are named after him due to this work \cite{steiner1853}. Oriented Steiner triple systems, the focus of this paper, were introduced 160 years later in the 2013 paper of Strambach and Stuhl \cite{StSt}.

In this paper, we explore new applications of oriented Steiner triple systems. We use the orientation to define a binary operation, on the real vector space spanned by the elements of a Steiner triple system, which we call the \textit{Steiner product}. The Steiner product generally satisfies some of the properties of the cross product of a vector space, as shown in Proposition \ref{perp}; and in two circumstances it satisfies all the properties.

We find that different orientations on the same Steiner triple system result in qualitatively different behavior from the Steiner product. This motivates the goal of Section 3 to classify oriented Steiner triple systems by isomorphism class. We succeed in this goal for Steiner systems of sizes 7 and 9, as shown in Theorems \ref{seven} and \ref{nine}. In addition, we compute the automorphism groups of each isomorphism class of these oriented Steiner triple systems. We end Section 3 with a look into a specific oriented Steiner triple system and its connection to the octonions.

In Section 4, we analyze the dynamics of the iterated Steiner product using representatives of the isomorphism classes found in Section 3. This culminates in Theorem \ref{thmdyn}, in which we give criteria for the dimension of the vector space spanned by the set of vectors generated from two starting vectors through an iterated Steiner product.
\section{Background and Definitions}
\begin{defn}
    A \textbf{Steiner triple system} is a special kind of block design. It consists of an ordered pair $(\S,T)$ where $\S$ denotes a finite set of \textbf{points} and $T$ denotes a collection of distinguished $3$-element subsets of $\S$ called \textbf{triples}. In a Steiner triple system, the set of triples, $T$, is constrained to satisfy the requirement that each $2$-element subset of $\S$ is a subset of exactly one of the triples in $T$ \cite{LR,StSt}.
\end{defn}
\begin{defn}
    An \textbf{oriented Steiner triple system} is a Steiner triple system, $(\S,T)$, where each triple, $\{a,b,c\}\in T$, is assigned one of the two possible cyclic orderings of the elements in the triple. We denote the cyclic ordering $a\rightarrow b \rightarrow c\rightarrow a$ by $[a,b,c]$ and $a\rightarrow c \rightarrow b\rightarrow a$ by $[a,c,b]$. Note that $[a,b,c]=[b,c,a]=[c,a,b]$ and that $[a,c,b]=[c,b,a]=[b,a,c]$. An oriented Steiner triple system, supported on $(\S,T)$, will be denoted by $(\S, \mathcal{O}(T))$.
\end{defn}
\begin{defn}
    Given an oriented Steiner triple system $(\S,\mathcal{O}(T))$, suppose $t=\{x_1,x_2,x_3\} \in T$ has the orientation $[x_1,x_2,x_3] \in \mathcal{O}(T)$. The oriented triple $[x_1,x_2,x_3]\in \O(T)$ can be used to define a skew symmetric function $f_t:\{x_1,x_2,x_3\}\times \{x_1,x_2,x_3\}\to\{-1,0,1\}$ by setting $$f_t(x_1,x_1)=f_t(x_2,x_2)=f_t(x_3,x_3)=0$$ $$f_t(x_1,x_2)=f_t(x_2,x_3)=f_t(x_3,x_1)=1$$ $$f_t(x_2,x_1)=f_t(x_3,x_2)=f_t(x_1,x_3)=-1$$

    The oriented triples in $\mathcal{O}(T)$ lead to functions $f_t$ for each $t\in T$. These can be combined to give a function $f:\S\times\S\to\{-1,0,1\}$ satisfying $f|_{t\times t}=f_t$ for each $t\in T$. We call $f$ the  \textbf{orientation function} for $(\S,\O(T))$. 
\end{defn}
Note that this definition is modified from that given in \cite{StSt}, where $f(x_i,x_i)=\pm 1$. Our having $f(x_i,x_i)=0$ is necessary for the construction of the Steiner product below.
\begin{defn}
    Let $(\S,T)$ be a Steiner triple system. Denote by $\R^{\S}$ the vector space of $\R$-linear formal sums of the elements $s\in\S$. That is, $$\R^{\S}=\left\{\sum_{s\in\S}a_ss:a_s\in \R\right\}.$$ Note that $\S$ is a basis for $\R^{\S}$ and that $\dim_\R\R^{\S}=|\S |$.
    
    We make $\R^{\S}$ into an inner product space by first defining the inner product on the basis $\S$ by $$\langle s,s'\rangle = \begin{cases}
        0&s\neq s'\\
        1&s=s'\\
    \end{cases}$$ then extending linearly to all of $\R^{\S}$. We will often denote the inner product with the standard dot product notation: $\langle s,s'\rangle=s\cdot s'$.
\end{defn}

\begin{defn}\label{StP}
Let $(\S,\O(T))$ be an oriented Steiner triple system with orientation function $f$. Consider the skew symmetric binary operation $\times:\S \times \S \to \R^{\S}$ given by $$s\times s'=f(s,s')s''$$ where $\{s,s',s''\}$ form a triple (or a subset of a triple in the case $s=s'$, in which case $f(s,s')=0$ so $s''$ does not matter). 

Using this binary operation on $\S$, we extend the binary operation linearly to all of $\R^{\S}$ to define the \textbf{Steiner product} on $\R^{\S}$. In particular, if  $\mathbf{a}=\sum_{s\in\S}a_ss\hskip .1in {\rm and} \hskip .1in  \mathbf{b}=\sum_{s\in\S}b_{s}s$ then  $$\mathbf{a}\times\mathbf{b}=\sum_{s\in\S}\sum_{t\in\S}a_sb_t(s\times t).$$
\end{defn}

 \begin{prop}{\label{perp}} Consider an oriented Steiner triple system 
$(\S,\O(T))$ with associated Steiner product $\times: \R^{\S} \times \R^{\S} \to \R^{\S}$. For any $\mathbf{a},\mathbf{b} \in \R^{\S}$, we have $\mathbf{a}\cdot(\mathbf{a}\times\mathbf{b})=\mathbf{b}\cdot(\mathbf{a}\times\mathbf{b})=0$.
\end{prop}

\begin{proof}
    Let $\S=\{s_1, \dots, s_n\}$ and write $\mathbf{a}=a_1s_1+\cdots+a_ns_n$. We have $$\mathbf{a}\cdot(\mathbf{a}\times\mathbf{b})=(a_1s_1+\cdots+a_ns_n)\cdot(\mathbf{a}\times\mathbf{b})=a_1s_1\cdot(\mathbf{a}\times\mathbf{b})+\cdots +a_ns_n\cdot(\mathbf{a}\times\mathbf{b}).$$ The product $a_is_i\cdot(\mathbf{a}\times\mathbf{b})$ is a sum of terms of the form $f(s_j,s_k)a_ia_jb_k$ where $\{s_i,s_j,s_k\}\in T$. This term is canceled out by the term $f(s_i,s_k)a_ja_ib_k$ in the product $a_js_j\cdot(\mathbf{a}\times\mathbf{b})$ (since $f(s_j,s_k)=-f(s_i,s_k)$ by the definition of $f$ in an oriented Steiner system). 
Therefore every term in $\mathbf{a}\cdot(\mathbf{a}\times\mathbf{b})$ is canceled by another term in the sum, and so $\mathbf{a}\cdot(\mathbf{a}\times\mathbf{b})=0$.
In a similar manner, $\mathbf{b}\cdot(\mathbf{a}\times\mathbf{b})=0$.
\end{proof}

\begin{defn}\label{cross}
    Given a real inner product space $V$ with inner product $\langle\cdot,\cdot\rangle$, a \textbf{cross product} on $V$ is a function $\cdot\times\cdot:V\times V\to V$ that satisfies the following three properties:
    \begin{enumerate}
        \item $\times$ is bilinear,
        \item $\langle v,v\times w \rangle=0$ for all $v,w\in V$,
        \item $|v|^2|w|^2=|v \times w|^2+\langle v,w\rangle^2$ for all $v,w\in V \ ($where $|v|^2=\langle v, v\rangle )$ \normalfont{\cite{WSM}}.
    \end{enumerate}
\end{defn}

\begin{rmk}
    The Steiner product on $\R^\S$ given in Definition \ref{StP} satisfies Criteria 1 and 2 of the cross product as given in Definition \ref{cross}.
\end{rmk}
In general, an oriented Steiner triple system leads to a Steiner Product that does not satisfy Criterion 3, with only two exceptions: an oriented Steiner triple system with three elements (with either of the two orientations), and the one with seven elements whose orientation corresponds to the multiplication table for distinct imaginary octonions.

\begin{prop}
    The Steiner product is skew-symmetric.
\end{prop}
\begin{proof}
    First note that by construction that if $s_i,s_j \in \S$ then $$s_i\times s_j=-(s_j\times s_i)$$ for all $i,j$. Now let $\mathbf{a}=\sum_{i=1}^na_is_i$ and $\mathbf{b}=\sum_{i=1}^nb_is_i$. Then $$\mathbf{a}\times\mathbf{b}=\sum_{i=1}^n\sum_{j=1}^na_ib_j(s_i\times s_j)=\sum_{i=1}^n\sum_{j=1}^n-a_ib_j(s_j\times s_i)=-\sum_{i=1}^n\sum_{j=1}^na_ib_j(s_j\times s_i)=-(\mathbf{b}\times\mathbf{a}).$$
\end{proof}

We will end this section with a word on denoting oriented Steiner triples. Consider the non-oriented Steiner triple system on seven elements, given by $$\S=\{s_1,s_2,s_3,s_4,s_5,s_6,s_7\},$$ $$T=\{\{s_1,s_2,s_3\},\{s_1,s_4,s_5\},\{s_1,s_6,s_7\},\{s_2,s_4,s_6\},\{s_2,s_5,s_7\},\{s_3,s_4,s_7\},\{s_3,s_5,s_6\}\}.$$ This Steiner triple system can be illustrated using the Fano plane, as shown in Figure \ref{fig:sts7}. Each of the seven lines of the Fano plane represents one of the triples of $T$.
    \begin{figure}[h]
        \centering
        \includegraphics[width=0.5\linewidth]{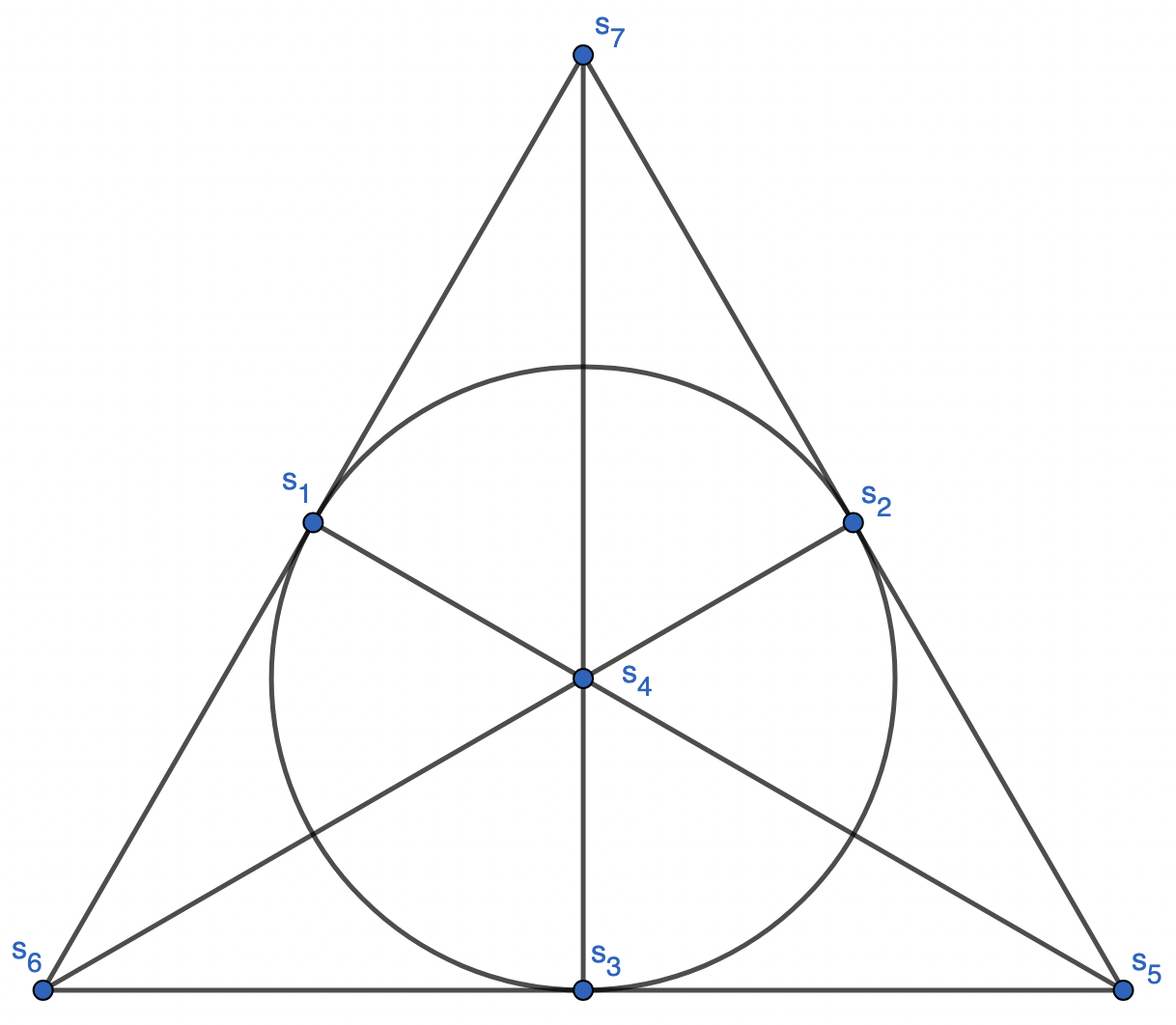}
        \caption{$STS(7)$ represented as the Fano plane}
        \label{fig:sts7}
    \end{figure}
    Now, if we want to orient the Steiner triple system $(\S,T)$, we can graphically represent each oriented triple as a directed cycle graph on three vertices. For example, we can represent the oriented triple $s_1\rightarrow s_2 \rightarrow s_3 \rightarrow s_1$ by a directed cycle graph, as shown in Figure \ref{fig:dcg}, 
    and illustrate the entire oriented Steiner triple system by placing arrows on all the lines in the Fano plane figure, as shown in Figure \ref{fig:octo} (with implied "wrapping arrows" that are not drawn).
    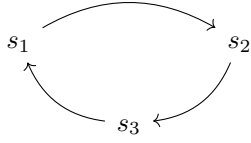
\begin{figure}\begin{center}
        \begin{tikzcd}
            s_1\arrow[rr, bend left]& &s_2\arrow[dl, bend left]\\
            &s_3\arrow[ul, bend left]& \\
        \end{tikzcd}
        \caption{Directed cycle graph}
        \label{fig:dcg}
    \end{center}
    \end{figure}
    
    \begin{figure}
        \centering
        \includegraphics[width=0.5\linewidth]{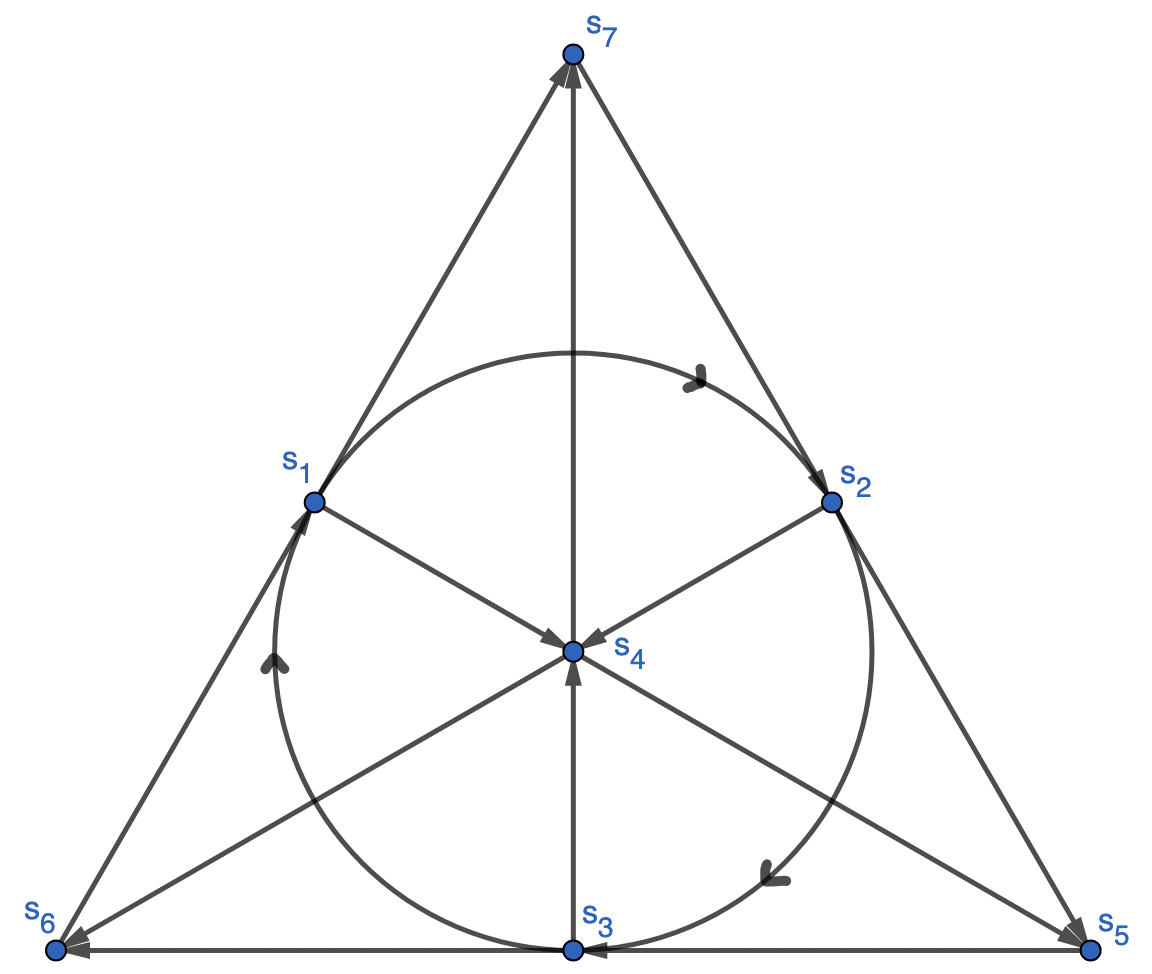}
        \caption{An oriented $STS(7)$}
        \label{fig:octo}
    \end{figure}
    
    Throughout this paper, we will use rectangular brackets as in $[s,s',s'']$ to denote oriented triples (whose associated skew-symmetric orientation function satisfies $f(s,s')=f(s',s'')=f(s'',s)=1$). With this notation, $[s_1,s_2,s_3]=[s_2,s_3,s_1]=[s_3,s_1,s_2]$ all represent the same orientation of the triple $\{s_1,s_2,s_3\}$ as depicted in Figure \ref{fig:octo}. Following this notation, we can represent the oriented Steiner triple system $(\S,\O(T))$, illustrated in Figure \ref{fig:octo}, as $$\S=\{s_1,s_2,s_3,s_4,s_5,s_6,s_7\},$$ $$\O(T)=\{[s_1,s_2,s_3],[s_1,s_4,s_5],[s_1,s_7,s_6],[s_2,s_4,s_6],[s_2,s_5,s_7],[s_3,s_4,s_7],[s_3,s_6,s_5]\}.$$
    
    \section{Classifying Oriented Steiner Systems of Order 7 and 9}
It is known that a Steiner triple system must have order $6k+1$ or $6k+3$ for some $k$. Up to isomorphism, there are unique Steiner triple systems of order 3, order 7, and order 9. There is a rapid increase in the number of Steiner triple systems of higher order. In particular, there are 2 of order 13, 80 of order 15, 11,084,874,829 of order 19, and 14,796,207,517,873,771 of order 21 \cite{HO}. When we consider imposing the additional structure of an orientation on a Steiner triple system, there is again a classification problem up to isomorphism. In this section, we classify the oriented Steiner triple systems of order 3, 7, and 9. 

\begin{defn}
Two Steiner triple systems $(\S,T)$ and $(\S',T')$ are \textbf{isomorphic} if there is a bijective map $$\psi:\S\to\S'$$ such that $\{s,s',s''\}\in T$ if and only if $\{\psi(s),\psi(s'),\psi(s'')\}\in T'$.

    Two oriented Steiner triple systems $(\S,\O(T))$ and $(\S',\O'(T'))$ are \textbf{isomorphic} if there is a bijective map $$\phi:\S\to\S'$$ such that  $[s,s',s'']\in \O(T)$ if and only if $[\phi(s),\phi(s'),\phi(s'')]\in \O'(T')$.
\end{defn}

The classification of oriented Steiner triple systems of order 3 is straightforward. There is a unique Steiner triple system, $(\S,T)$, of order $3$ given by $\S=\{1,2,3\}$, $T=\{\{1,2,3\}\}$. Initially, it seems there are two oriented Steiner triple systems of order 3: $(\{1,2,3\},\{[1,2,3]\})$ and $(\{1,2,3\},\{[1,3,2]\})$. However, using the bijective map $\phi: \{1,2,3\} \rightarrow \{1,2,3\}$ given by $\phi(1)= 1, \phi(2)=3, \phi(3)=2$, we get an isomorphism between $(\{1,2,3\},\{[1,2,3]\})$ and $(\{1,2,3\},\{[1,3,2]\})$. Thus, there is a unique oriented Steiner triple system of order 3 up to isomorphism.

\begin{defn}
    Let $(\S,\O(T))$ be an oriented Steiner triple system. We can build an oppositely oriented Steiner triple system by reversing the orientation of each element of $\O(T)$. We denote this oppositely oriented system by $(\S,\overline{\O}(T))$. We will call $(\S,\O(T))$ \textbf{reflexive} if $(\S,\O(T))$ is isomorphic to $(\S,\overline{\O}(T))$.
\end{defn}

Up to isomorphism, there is only one oriented Steiner triple systems of order 3 and it is reflexive.
We now focus on the classification of oriented Steiner triple systems of order 7 and determine their automorphism groups.

\begin{defn}
    The \textbf{automorphism group} of a Steiner triple system $(\S,T)$ is the group (under composition) of isomorphisms $\phi: (\S,T) \rightarrow (\S,T)$. We denote this by $\Aut(\S,T)$. Similarly, the automorphism group of an oriented Steiner triple system $(\S,\O(T))$ is the group of isomorphisms $\phi: (\S,\O(T)) \rightarrow (\S,\O(T))$. We denote this by $\Aut(\S,\O(T))$.
\end{defn}

\begin{rmk}
    Let $(\S,T)$ be a Steiner triple system and let $(\S,\O(T))$ be an oriented Steiner triple system. Suppose $\phi:\S \rightarrow \S$ is bijective. Define $$\phi(T)=\{\{\phi(s),\phi(s'),\phi(s'')\}\ |\ \{s,s',s''\} \in T\}$$ $$\phi(\O(T))=\{[\phi(s),\phi(s'),\phi(s'')]\ |\ [s,s',s''] \in \O(T)\}.$$ We have $$\Aut(\S,T) = \{\phi:\S \rightarrow \S\ |\ \phi \ {\rm is\ bijective\ and\ } \phi(T)=T\}$$ $$\Aut(\S,\O(T)) = \{\phi:\S \rightarrow \S\ |\ \phi \ {\rm is\ bijective\ and\ } \phi(\O(T))=\O(T)\}.$$ From this observation, it is clear that $\Aut(\S,T)$ and $\Aut(\S,\O(T))$ have natural representations as subgroups of the symmetric group $Aut(\S)=S_{\S}$ and that the representation of $\Aut(\S,\O(T))$ sits as a subgroup in the representation of $\Aut(\S,T)$. We will denote this by $\Aut(\S,\O(T)) < \Aut(\S,T)$. In the order 3 case,  we have $\Aut(\{1,2,3\},\{[1,2,3]\}) \cong \Aut(\{1,2,3\},\{[2,1,3]\}) \cong C_3 < \Aut(\{1,2,3\},\{\{1,2,3\}\})\cong  S_3$. In terms of their permutation representation, we have $\langle (1,2,3)\rangle < \langle (1,2,3),(1,2)\rangle = S_3$ $($where $\langle (1,2,3),(1,2)\rangle$ denotes the subgroup generated by the permutations, in cycle notation, $(1,2,3)$ and $(1,2))$. As noted before, we have an isomorphism between the oriented Steiner triple systems $(\{1,2,3\},\{[1,2,3]\})$ and $(\{1,2,3\},\{[1,3,2]\})$ given by the transposition $(2,3)$.
\end{rmk}

An interesting theorem of Strambach and Stuhl shows that the automorphism group of an oriented Steiner triple system must have odd order \cite{StSt}. A simple illustration of this theorem is seen in the order 3 case where $\Aut(\{1,2,3\},\{[1,2,3]\}) \cong \Aut(\{1,2,3\},\{[2,1,3]\}) \cong C_3$. We will see further illustrations of this theorem for oriented Steiner triple systems of order 7 and order 9 in the rest of this section.

 Up to isomorphism, there is a unique Steiner triple system on seven elements. For clarity, we will fix a model of this system. Instead of using variables to label the seven elements in the set, we will use the numbers $1,2,3,4,5,6,7$. The collection of triples will therefore be triples of numbers. For the theorem that follows, we use the following model:
 $$\S=\{1,2,3,4,5,6,7\},$$ $$T=\{\{1,2,3\},\{1,4,5\},\{1,6,7\},\{2,4,6\},\{2,5,7\},\{3,4,7\},\{3,5,6\}\}.$$

 It is well known that the automorphism group of the Steiner triple system, which we can represent as a subgroup of $S_7$, is a simple group of order $168$.
 
\begin{thm}\label{seven}
There are exactly four oriented Steiner triple systems on seven elements, up to isomorphism. Two isomorphism classes have a non-Abelian automorphism group of order 21, and two isomorphism classes have a cyclic automorphism group of order 3. None of the isomorphism classes are reflexive.

Representatives of the distinct isomorphism classes are $$(\S,\O_1(T))=\{[1, 2, 3], [1, 4, 5], [1, 6, 7], [2, 4, 6], [2, 7, 5], [3, 6, 5], [3, 7, 4]\}_{21},$$ 
$$(\S,\O_2(T))=\{[1, 2, 3], [1, 4, 5], [1, 7, 6], [2, 4, 6], [2, 5, 7], [3, 4, 7], [3, 6, 5]\}_{21},$$

$$(\S,\O_3(T))=\{[1, 2, 3], [1, 4, 5], [1, 6, 7], [2, 4, 6], [2, 5, 7], [3, 4, 7], [3, 5, 6]\}_{3},$$ 
$$(\S,\O_4(T))=\{[1, 2, 3], [1, 4, 5], [1, 6, 7], [2, 4, 6], [2, 5, 7], [3, 4, 7], [3, 6, 5]\}_{3}.$$

We have $(\S,\overline{\O_1}(T))\cong (\S, \O_2(T))$ and $(\S,\overline{\O_3}(T))\cong (\S, \O_4(T))$ (where $\overline{\O}(T)$ denotes that the orientation of each oriented triple in $\O(T)$ has been reversed). The subscript on each isomorphism class denotes the order of its automorphism group. Namely, $|\Aut(\S,\O_1(T))|=|\Aut(\S,\O_2(T))|=21$ and $|\Aut(\S,\O_3(T))|=|\Aut(\S,\O_4(T))|=3$.
\end{thm}
\begin{proof}
The proof is given through an explicit computation in $Maple^{TM}$ \cite{Maple}.
The Maple code for the computation can be found in Appendix \ref{sts237}.
From the Maple code and the use of the open source computer algebra system \textbf{GAP} \cite{GAP4}, we have the permutation representation $\Aut(\S,T) \cong \langle (1,2,4,3,6,7,5),(4,5)(6,7)\rangle < \Aut(\S)=S_7$. GAP identifies this representation of $\Aut(\S,T)$ as a simple group of order 168.
The permutation representation of $\Aut(\S,\O_1(T))$ as a subgroup of $\langle (1,2,4,3,6,7,5),(4,5)(6,7)\rangle$ is $$\Aut(\S,\O_1(T))=\langle(2,4,6)(3,5,7),(1,2,3)(4,7,6)\rangle.$$ GAP identifies this group as $ C_7\rtimes C_3$, a non-Abelian group of order 21. The permutation representation of $\Aut(\S,\O_2(T))$ as a subgroup of $\langle (1,2,4,3,6,7,5),(4,5)(6,7)\rangle$ is $$\Aut(\S,\O_2(T))=\langle(2,4,7)(3,5,6),(1,2,3)(5,6,7)\rangle.$$ GAP identifies this group as $ C_7\rtimes C_3$, the same non-Abelian group of order 21 as found for $\Aut(\S,\O_1(T))$. These are two different permutation representations of the same group. The permutation representations are conjugate within the permutation representation of $\Aut(\S,T)$. The permutation representation of $\Aut(\S,\O_3(T))$ as a subgroup of $\langle (1,2,4,3,6,7,5),(4,5)(6,7)\rangle$ is $\Aut(\S,\O_3(T))=\langle(2,4,6)(3,5,7)\rangle.$ The permutation representation of $\Aut(\S,\O_4(T))$ as a subgroup of $\langle (1,2,4,3,6,7,5),(4,5)(6,7)\rangle$ is $\Aut(\S,\O_4(T))=\langle(1,7,6)(3,5,4)\rangle.$ Both of the automorphism groups are isomorphic to the cyclic group of order 3. Their permutation representations are conjugate in the group $\langle (1,2,4,3,6,7,5),(4,5)(6,7)\rangle$.
\end{proof}
Note that since $\Aut(\S,\O_1(T))$ and $\Aut(\S,\O_3(T))$ are subgroups of the simple group $\Aut(\S,T)$, neither are normal in $\Aut(\S,T)$. Another model for the simple group $\Aut(\S,T)$ is $\text{GL}(3,\F_2)$ \cite{Brown01102009}.

Up to isomorphism, there is a unique Steiner triple system, $(\S,T)$ of order nine. In condensed notation, we will use the following representation of this system:

$(\S,T) = (\{1,2,3,4,5,6,7,8,9\},\  \{123,456,789,147,258,369,159,267,348,168,249,357\}).$

 The automorphism group of this Steiner triple system has a permutation representation as the subgroup of $S_9$ given by $\langle (2,6,4,9,3,8,7,5),(1,3,2)(4,7,5,8,6,9)\rangle$. It is a group of order $432$.
 
\begin{thm}\label{nine}
There are exactly 16 isomorphism classes of oriented Steiner triple systems on nine elements. Seven of the classes have an automorphism group of size 1, seven have an automorphism group of size 3, one of size 9, and one of size 27. In the list below, we will denote the oriented Steiner triple systems as $(\S,\O_1(T)), \dots, (\S,\O_{16}(T))$. The first eight isomorphism classes, $(\S,\O_1(T)), \dots, (\S,\O_{8}(T))$ given below are reflexive. The remaining eight are not reflexive and are grouped as four pairs. In other words, we have $(\S,\O_9(T))\cong (\S,\overline{\O}_{10}(T)),(\S,\O_{11}(T))\cong (\S,\overline{\O}_{12}(T)),(\S,\O_{13}(T))\cong (\S,\overline{\O}_{14}(T)),(\S,\O_{15}(T))\cong (\S,\overline{\O}_{16}(T))$.

Representatives of the distinct isomorphism classes are: $$\{[1, 2, 3], [1, 4, 7], [1, 5, 9], [1, 6, 8], [2, 4, 9], [2, 5, 8],[2, 6, 7], [3, 4, 8], [3, 5, 7], [3, 6, 9], [4, 5, 6], [7, 8, 9]\}_{27},$$
$$\{[1,2,3],[1,4,7],[1,5,9],[1,6,8],[2,4,9],[2,5,8],[2,6,7],[3,4,8],[3,5,7],[3,6,9],[4,5,6],[7,9,8]\}_9,$$
$$\{[1,2,3],[1,4,7],[1,5,9],[1,6,8],[2,4,9],[2,5,8],[2,6,7],[3,4,8],[3,5,7],[3,9,6],[4,5,6],[7,8,9]\}_3,$$
$$\{[1,2,3],[1,4,7],[1,5,9],[1,6,8],[2,4,9],[2,5,8],[2,6,7],[3,7,5],[3,8,4],[3,9,6],[4,5,6],[7,8,9]\}_3,$$
$$\{[1,2,3],[1,4,7],[1,5,9],[1,6,8],[2,4,9],[2,5,8],[2,6,7],[3,7,5],[3,8,4],[3,9,6],[4,5,6],[7,9,8]\}_3,$$
$$\{[1, 2, 3], [1, 4, 7], [1, 5, 9], [1, 6, 8], [2, 4, 9], [2, 5, 8], [2, 6, 7], [3, 4, 8], [3, 5, 7], [3, 9, 6], [4, 5, 6], [7, 9, 8]\}_1,$$
$$\{[1, 2, 3], [1, 4, 7], [1, 5, 9], [1, 6, 8], [2, 4, 9], [2, 5, 8], [2, 6, 7], [3, 4, 8], [3, 7, 5], [3, 9, 6], [4, 5, 6], [7, 8, 9]\}_1,$$
$$\{[1, 2, 3], [1, 4, 7], [1, 5, 9], [1, 6, 8], [2, 4, 9], [2, 5, 8], [2, 6, 7], [3, 4, 8], [3, 7, 5], [3, 9, 6], [4, 6, 5], [7, 9, 8]\}_1,$$
  
  $$\{[1,2,3],[1,4,7],[1,5,9],[1,6,8],[2,4,9],[2,5,8],[2,6,7],[3,7,5],[3,8,4],[3,9,6],[4,5,6],[7,9,8]\}_3,$$ 
$$\{[1,2,3],[1,4,7],[1,5,9],[1,6,8],[2,4,9],[2,5,8],[2,6,7],[3,7,5],[3,8,4],[3,9,6],[4,6,5],[7,8,9]\}_3,$$ 

$$\{[1,2,3],[1,4,7],[1,5,9],[1,6,8],[2,4,9],[2,5,8],[2,7,6],[3,4,8],[3,7,5],[3,9,6],[4,5,6],[7,8,9]\}_3,$$
  $$\{[1, 2, 3], [1, 4, 7], [1, 5, 9], [1, 6, 8], [2, 4, 9], [2, 7, 6], [2, 8, 5], [3, 4, 8], [3, 5, 7], [3, 9, 6], [4, 6, 5], [7, 9, 8]\}_3,$$
  
  $$\{[1, 2, 3], [1, 4, 7], [1, 5, 9], [1, 6, 8], [2, 4, 9], [2, 5, 8], [2, 6, 7], [3, 4, 8], [3, 7, 5], [3, 9, 6], [4, 5, 6], [7, 9, 8]\}_1,$$ 
  $$\{[1, 2, 3], [1, 4, 7], [1, 5, 9], [1, 6, 8], [2, 4, 9], [2, 5, 8], [2, 6, 7], [3, 4, 8], [3, 7, 5], [3, 9, 6], [4, 6, 5], [7, 8, 9]\}_1,$$ 
  
  $$\{[1, 2, 3], [1, 4, 7], [1, 5, 9], [1, 6, 8], [2, 4, 9], [2, 5, 8], [2, 7, 6], [3, 4, 8], [3, 7, 5], [3, 9, 6], [4, 5, 6], [7, 9, 8]\}_1,$$ 
  $$\{[1, 2, 3], [1, 4, 7], [1, 5, 9], [1, 6, 8], [2, 4, 9], [2, 5, 8], [2, 7, 6], [3, 4, 8], [3, 7, 5], [3, 9, 6], [4, 6, 5], [7, 9, 8]\}_1.$$
\end{thm}

\begin{proof}
The proof is given via an explicit computation in $Maple$ \cite{Maple}.
The Maple code for the computation can be found in Appendix \ref{sts239}.
From the Maple code and the use of the open source computer algebra system \textbf{GAP} \cite{GAP4}, we find the permutation representation $$\Aut(\S,T) \cong \langle (2,6,4,9,3,8,7,5),(1,3,2)(4,7,5,8,6,9)\rangle < \Aut(\S).$$ GAP identifies this representation of $\Aut(\S,T)$ with the identifier $[432,734]$. This is a group of order $432$ and corresponds to the affine group $Aff(2,\mathbb{F}_3)$. In general, the affine group on a vector space $V$ can be expressed as $Aff(V)=V \rtimes GL(V)$. Thus, we can express $Aff(2,\mathbb{F}_3)$ as $\mathbb{F}_3^2\rtimes GL(2,\mathbb{F}_3)$.
The permutation representation of $\Aut(\S,\O_1(T))$ as a subgroup of the permutation representation of $\Aut(\S,T)$ is $$\Aut(\S,\O_1(T))=\langle(4,5,6)(7,9,8),(1,4,9)(2,5,7)(3,6,8)\rangle.$$ The GAP identifier of this group is $[27,3]$, a non-Abelian group of order $27$ and exponent $3$. This group can be expressed as a semi-direct product $C_3^2\rtimes C_3$. A model for this group is as the set of upper triangular $3\times 3$ matrices with entries from $\mathbb{F}_3$ and $1's$ on the diagonal with matrix multiplication as the binary operation (the Heisenberg group $\text{He}_3$ over $\Z/3\Z$). The permutation representation of $\Aut(\S,\O_2(T))$ as a subgroup of the permutation representation of $\Aut(\S,T)$ is $$\Aut(\S,\O_2(T))=\langle(4,5,6)(7,9,8),(1,2,3)(7,9,8)\rangle.$$ GAP identifies this group as $ C_3\times C_3$. The permutation representations of $\Aut(\S,\O_i(T))$ for $i=3,4,5,9,10,11,12$ as  subgroups of the permutation representation of $\Aut(\S,T)$ are 

\

$\Aut(\S,\O_3(T))=\langle(1,7,4)(2,8,5)(3,9,6)\rangle$

$\Aut(\S,\O_4(T))=\langle(4,5,6)(7,9,8)\rangle$

$\Aut(\S,\O_5(T))=\langle(4,5,6)(7,9,8)\rangle$

$\Aut(\S,\O_9(T))=\langle(4,5,6)(7,9,8)\rangle$

$\Aut(\S,\O_{10}(T))=\langle(4,5,6)(7,9,8)\rangle$

$\Aut(\S,\O_{11}(T))=\langle(1,8,4)(2,9,5)(3,7,6)\rangle$

$\Aut(\S,\O_{12}(T))=\langle(1,2,7)(3,6,4)(5,8,9)\rangle$

\end{proof}

It is known that $\Aut(\S,T)\cong Aff(2,\F_3)$  has six normal subgroups (up to conjugacy class), none of which are $\text{He}_3$, $C_3^2$, or $C_3$. Thus only the oriented Steiner systems with trivial automorphism groups have automorphism groups that are normal in $\Aut(\S,T)$.

Some of the isomorphism classes are simply the opposite orientations of each other. Inverting the orientation of an oriented Steiner triple system will not necessarily result in an isomorphic oriented Steiner triple system, but will result in one with an isomorphic automorphism group (in fact, with an identical permutation representation).
\begin{prop}
    Let $(\S,\O(T))$ be an oriented Steiner triple system, and let $(\S,\overline{\O}(T))$ be the oriented Steiner triple system obtained by reversing all oriented triples in $(\S,\O(T))$. Then $\Aut(\S,\O(T))=\Aut(\S,\overline{\O}(T))$.
\end{prop}
\begin{proof}
    Let $f\in \Aut(\S,\O(T))$ and let $[s,s',s'']\in\overline{\O}(T)$. Then we have $[s,s'',s']\in\O(T)$, and so $[f(s),f(s''),f(s')]\in\O(T)$. Thus $[f(s),f(s'),f(s'')]\in\overline{\O}(T)$ and so $f\in\Aut(\S,\overline{\O}(T))$. The other direction is proven the same way.
\end{proof}

Suppose two oriented Steiner triple systems are isomorphic via an element $g\in S_{\S}$. The automorphism groups of the two oriented Steiner triple systems will be isomorphic and will both be subgroups of $\Aut(\S,T)$ but the explicit permutation representations of the automorphism groups may differ. They will be conjugate in $\Aut(\S,T)$ via conjugation by the element $g$.

\vspace{\baselineskip}

We would like to end this section with a look into how an automorphism of an oriented Steiner triple system can be lifted to an automorphism of $\R^\S$, and what special subgroups of $\Aut(\R^\S)$ can be identified therefrom. Let $\sigma\in\Aut(\S)$. We can lift the action of $\sigma$ to $\R^\S$. In particular, if $a=\sum_{s\in\S}a_s s \in\R^\S$, then define $$\sigma(a)=\sum_{s\in\S}a_s\sigma(s).$$ This action of $\sigma$ induces an automorphism of $\R^\S$ and we have the following sequence of injective group homomorphisms $$\Aut(\S,\O(T))\hookrightarrow \Aut(\S,T)\hookrightarrow \Aut(\S) \hookrightarrow\Aut(\R^\S).$$ Through these maps, we use elements of $\Aut(\S,\O(T)),\ \Aut(\S,T)$, and $\Aut(\S)$ to act on vectors in $\R^\S$.
\begin{lemma}
    Let $\sigma\in\Aut(\S,\O(T))$, and let $a,b\in\R^{\S}$. Then $\sigma(a\times b)=\sigma(a)\times\sigma(b)$.
\end{lemma}
\begin{proof}
    First, let us prove the claim for $s_i,s_j\in\S$. Let $[s_i,s_j,s_k]$ be an oriented triple in the system. Then $[\sigma(s_i),\sigma(s_j),\sigma(s_k)]$ is also an oriented triple in $(\S,\O(T))$. Thus $s_i\times s_j=s_k$ and $\sigma(s_i)\times\sigma(s_j)=\sigma(s_k)=\sigma(s_i\times s_j)$. Similarly, $\sigma(s_j\times s_i)=-\sigma(s_k)=\sigma(s_j)\times\sigma(s_i)$.
By bilinearity of the Steiner product and linearity of $\sigma$, the claim is also true for general linear combinations $a,b\in\R^{\S}$.
\end{proof}
 Thus every element of $\Aut(\S,\O(T))$ is an automorphism of $\R^\S$ that commutes with the Steiner product. We let $\Aut(\R^\S,\times)$ denote the collection of all automorphisms of $\R^\S$ that commute with the Steiner product. We have $\Aut(\S,\O(T))\leq\Aut(\R^\S,\times)$.

 \begin{thm}
     Let $(\S,\O_1(T))$ be the oriented Steiner triple system from Theorem \ref{seven}. More specifically, we consider $$(\S,\O_1(T))=\{[s_1,s_2,s_3],[s_1,s_4,s_5],[s_1,s_6,s_7],[s_2,s_4,s_6],[s_2,s_7,s_5],[s_3,s_6,s_5],[s_3,s_7,s_4]\}.$$ Let $\Aut(\R^\S,\times)$ denote the automorphisms of $\R^\S$ that commute with the associated Steiner product. We have $\Aut(\R^\S,\times)\cong G_2$, where $G_2$ is the smallest exceptional Lie group.
 \end{thm}
 \begin{proof}
     First note that the octonions $\OO=\R\oplus \R^\S$ have the following multiplication, denoted by $*$: $$(\alpha+\mathbf{a})*(\beta+\mathbf{b})=\alpha\beta+\alpha\mathbf{b}+\beta\mathbf{a}+\mathbf{a}*\mathbf{b}=\alpha\beta+\alpha\mathbf{b}+\beta\mathbf{a}+\mathbf{a}\times \mathbf{b}-\langle\mathbf{a},\mathbf{b}\rangle=[\alpha\beta-\langle\mathbf{a},\mathbf{b}\rangle]+[\alpha\mathbf{b}+\beta\mathbf{a}+\mathbf{a}\times \mathbf{b}],$$ where $\alpha\beta-\langle\mathbf{a},\mathbf{b}\rangle\in\Re(\OO)$ and $\alpha\mathbf{b}+\beta\mathbf{a}+\mathbf{a}\times \mathbf{b}\in\Im(\OO)$ \cite{WSM}.


     We will use the fact that $\Aut(\OO)\cong G_2$ \cite{baez}. Let $F\in\Aut(\OO)$. Then $F(1)=1$ and $F|_{\Im(\OO)}\in\text{SO}(\Im(\OO))$ \cite{baez}. Denote by $f$ this restriction $F|_{\Im(\OO)}$. Then we will show that $f\in\Aut(\R^\S,\times)$.

Observe \begin{align*}f(\mathbf{a}\times \mathbf{b})=F(\mathbf{a}* \mathbf{b}+\langle\mathbf{a},\mathbf{b}\rangle)=f(\mathbf{a})* f(\mathbf{b})+\langle\mathbf{a},\mathbf{b}\rangle=f(\mathbf{a})\times f(\mathbf{b})-\langle f(\mathbf{a}),f(\mathbf{b})\rangle+\langle\mathbf{a},\mathbf{b}\rangle.\end{align*}

Since $f\in\text{SO}(\R^\S)$, we know $\langle f(\mathbf{a}),f(\mathbf{b})\rangle=\langle\mathbf{a},\mathbf{b}\rangle$ and so $$f(\mathbf{a})\times f(\mathbf{b})-\langle f(\mathbf{a}),f(\mathbf{b})\rangle+\langle\mathbf{a},\mathbf{b}\rangle=f(\mathbf{a})\times f(\mathbf{b})$$ and so we have $$f(\mathbf{a}\times \mathbf{b})=f(\mathbf{a})\times f(\mathbf{b})$$ and thus $f\in\Aut(\R^\S,\times)$.

Now let $g\in\Aut(\R^\S,\times)$. Define $G:\OO=\R\oplus\R^\S\to\OO=\R\oplus \R^\S$ as $$G:=\begin{pmatrix}
    1&\mathbf{0}\\
    \mathbf{0}&g
\end{pmatrix}.$$ Then we will show that $G\in\Aut(\OO)$.
Observe \begin{align*}G((\alpha+\mathbf{a})*(\beta+\mathbf{b}))&=G(\alpha\beta+\alpha \mathbf{b}+\beta \mathbf{a}+\mathbf{a}* \mathbf{b})\\&=\alpha\beta+\alpha g(\mathbf{b})+\beta g(\mathbf{a})+g(\mathbf{a}\times \mathbf{b})-\langle\mathbf{a},\mathbf{b}\rangle.\end{align*}

Now note that \begin{align*}G(\alpha+\mathbf{a})* G(\beta+\mathbf{b})&=(\alpha+g(\mathbf{a}))*(\beta+g(\mathbf{b}))\\&=\alpha\beta+\alpha g(\mathbf{b})+\beta g(\mathbf{a})+g(\mathbf{a})* g(\mathbf{b})\\&=\alpha\beta+\alpha g(\mathbf{b})+\beta g(\mathbf{a})+g(\mathbf{a})\times g(\mathbf{b})-\langle g(\mathbf{a}),g(\mathbf{b})\rangle\\&=\alpha\beta+\alpha g(\mathbf{b})+\beta g(\mathbf{a})+g(\mathbf{a}\times\mathbf{b})-\langle g(\mathbf{a}),g(\mathbf{b})\rangle.\\
\end{align*}

Next we will show that $\langle\mathbf{a},\mathbf{b}\rangle=\langle g(\mathbf{a}),g(\mathbf{b})\rangle$. Let $s\neq s'\in\S$ be basis vectors for $\R^\S$ where $[s,s',s'']\in\O(T)$. Then $$\langle g(s),g(s')\rangle=\langle g(s),g(s''\times s)\rangle=\langle g(s),g(s'')\times g(s)\rangle=0$$ by Proposition \ref{perp}. 

Next we must show that $\langle g(s),g(s)\rangle=1$. First let $g(s)=\ell \mathbf{u}$ and $g(s')=\ell'\mathbf{u}'$ where $\ell,\ell'\in\R^+$ and $\langle \mathbf{u},\mathbf{u}\rangle=\langle \mathbf{u}',\mathbf{u}'\rangle=1$. As we've just seen, $\langle \mathbf{u},\mathbf{u}'\rangle=0$. Then $g(s'')=g(s\times s')=\ell\ell'(\mathbf{u}\times\mathbf{u}')$. Since the multiplication of the octonions induces a genuine cross product on $\R^7$ \cite{WSM}, we have access to the third criterion of cross products from Definition \ref{cross}: $$|\mathbf{u}\times\mathbf{u}'|=|\mathbf{u}|^2|\mathbf{u}'|^2-\langle \mathbf{u},\mathbf{u}'\rangle^2=1-0=1.$$ Therefore $$|g(s'')|=\ell\ell'=|g(s)||g(s')|.$$

Similarly, we have \begin{align*}
    |g(s)|&=|g(s')||g(s'')|,\\
    |g(s')|&=|g(s)||g(s'')|.
\end{align*} The only solution to all three of these equations in $(\R^+)^3$ is $$|g(s)|=|g(s')|=|g(s'')|=1.$$ Therefore $\langle g(s),g(s)\rangle=1$ for all $s\in\S$. Together with $\langle g(s),g(s')\rangle=0$ for all $s\neq s'\in\S$ and linearity, we have $\langle g(\mathbf{a}),g(\mathbf{b})\rangle=\langle\mathbf{a},\mathbf{b}\rangle$ for all $\mathbf{a},\mathbf{b}\in\R^\S$.


Thus we have $$G((\alpha+\mathbf{a})*(\beta+\mathbf{b}))=G(\alpha+\mathbf{a})* G(\beta+\mathbf{b}),$$ and so we have $G\in\Aut(\OO)$. Since these two maps \begin{align*}
    F&\mapsto F|_{\R^\S},\\
    g&\mapsto\begin{pmatrix}
        1&\mathbf{0}\\
        \mathbf{0}&g
    \end{pmatrix}
\end{align*} are invertible group homomorphisms between $\Aut(\R^\S,\times)$ and $\Aut(\OO)$, we have $$\Aut(\R^\S,\times)\cong\Aut(\OO)\cong G_2.$$
 \end{proof}

\section{The Dynamics of the Steiner Product}
In this section, we will use the notation $ \pmb{\langle} \mathbf{v},\mathbf{w} \pmb{\rangle} $ to denote the vector space spanned by $\mathbf{v}$ and $\mathbf{w}$.
The Steiner product will generally fail the third criterion of the cross product. In this section, we explore this failure in greater detail and show that we can relate the dynamics of a certain iterated Steiner product with the spectral theorem for skew symmetric matrices. We illustrate the failure of the Steiner product to be a cross product in the following example.

\begin{ex}\label{zd}
    Take the oriented Steiner system on $\S=\{s_1,s_2,s_3,s_4,s_5,s_6,s_7\}$ given by $$\O(T)=\{[s_1,s_2,s_3],[s_1,s_4,s_5],[s_1,s_7,s_6],[s_2,s_4,s_6],[s_2,s_5,s_7],[s_3,s_4,s_7],[s_3,s_5,s_6]\}.$$

    A direct computation shows that $(s_1+s_5)\times(s_3+s_7)=-s_2+s_6-s_6+s_2=0.$ Therefore $$4=|s_1+s_5|^2|s_3+s_7|^2 \neq |(s_1+s_5)\times(s_3+s_7)|^2+|(s_1+s_5)\cdot(s_3+s_7)|^2=0.$$
\end{ex}
The automorphism group of $(\S,\O(T))$ is $\{\text{id},(274)(365),(247)(356)\}$. Thus the sets $\{s_1+s_5,s_1+s_3,s_1+s_6\}$ and $\{s_3+s_7,s_6+s_4,s_5+s_2\}$ form the orbits of $s_1+s_5$ and $s_3+s_7$ under $\Aut(\S,\O(T))$, respectively. Notice also that the vector space $\pmb{\langle} s_1+s_5,s_3+s_7,(s_1+s_5)\times (s_3+s_7)\pmb{\rangle}$ is 2-dimensional, while with a true cross product we would expect this vector space to be 3-dimensional (again, note that for this section, we will use the notation $ \pmb{\langle} \mathbf{v},\mathbf{w} \pmb{\rangle} $ to denote the vector space spanned by $\mathbf{v}$ and $\mathbf{w}$). This leads us to take a closer look at the dynamics of the iterated cross product with respect to the dimension of the resulting vector spaces. Given $\v \in \R^{\S}$, we let $[ \v ]$ denote the coordinate vector of $\v$ with respect to the ordered basis $s_1, s_2, \dots, s_n$.

 \begin{defn}
        Given an oriented Steiner system $(\S, \O(T))$, the vector $\v\in\R^\S$ is a \textbf{zero-divisor} if there is a $\w\in\R^\S$ such that $\pmb{\langle}\v\pmb{\rangle}\neq\pmb{\langle}\w\pmb{\rangle}$ and $\v\times\w=0$.
    \end{defn}
\begin{defn}
    Let $\v,\w\in \R^{\S}$. Define $L_{\w}:\R^{\S}\to \R^{\S}$ as $L_{\w}(\v)=\w\times \v$. Set $L^0_{\w}$ to be the identity map, set $L^1_{\w}=L_{\w}$, and define $L^{k+1}_{\w}(\v)=L_{\w}(L^k_{\w}(\v))$.
\end{defn}
\begin{prop}{\label{plateau}}
    Let $\v,\w\in \R^{\S}$. If $\dim\pmb{\langle} \v,L^1_{\w}(\v),\dots,L^{k}_{\w}(\v)\pmb{\rangle}=\dim\pmb{\langle} \v,L^1_{\w}(\v),\dots,L^{k+1}_{\w}(\v)\pmb{\rangle}$, then $\dim\pmb{\langle} \v,L^1_{\w}(\v),\dots,L^{k}_{\w}(\v)\pmb{\rangle}=\dim\pmb{\langle} \v,L^1_{\w}(\v),\dots,L^{k+n}_{\w}(\v)\pmb{\rangle}$ for all $n\in\N$.
\end{prop}
\begin{proof}
    If we have $\dim\pmb{\langle} \v,L^1_{\w}(\v),\dots,L^{k}_{\w}(\v)\pmb{\rangle}=\dim\pmb{\langle} \v,L^1_{\w}(\v),\dots,L^{k+1}_{\w}(\v)\pmb{\rangle}$ then we can conclude that $L^{k+1}_{\w}(\v)\in \pmb{\langle} \v,L_{\w}(\v),\dots,L^{k}_{\w}(\v)\pmb{\rangle}$. Thus we can write $$L^{k+1}_{\w}(\v)=b_0\v+b_1L^1_{\w}(\v)+\cdots+b_kL^k_{\w}(\v).$$ This implies that $$L^{k+2}_{\w}(\v)=L^{k+1}_{\w}(\v)\times \w=b_0L^1_{\w}(\v)+b_1L^2_{\w}(\v)+\cdots+b_kL^{k+1}_{\w}(\v).$$ Thus $$L^{k+2}_{\w}(\v)\in\pmb{\langle} \v,L_{\w}(\v),\dots,L^{k}_{\w}(\v)\pmb{\rangle}.$$ Continuing in this manner, we get $$L^{k+n}_{\w}(\v)\in\pmb{\langle} \v,L_{\w}(\v),\dots,L^{k}_{\w}(\v)\pmb{\rangle}$$ for all $n\in\N$.
\end{proof}
\begin{lemma}
    Let $\v,\w\in \R^{\S}$. Then $\v\times\w\in\pmb{\langle}\v,\w\pmb{\rangle}$ iff $\v\times\w=0$.
\end{lemma}
\begin{proof}
    The reverse direction is clear. Now let $\v\times\w=a\v+b\w$ for $a,b\in \R$. Then $|\v\times\w|^2=(\v\times\w)\cdot(\v\times\w)=(\v\times\w)\cdot(a\v+b\w)=a(\v\times\w)\cdot \v+b(\v\times\w)\cdot\w=0$, and so $\v\times\w=0$.
\end{proof}

    In the following proposition, given $\v=v_1s_1 + \dots + v_ns_n \in \R^{\S}$, we denote by $[\v]$ the column vector in $\R^n$ whose entries are the $s_i$-coefficients $v_1, v_2, \dots, v_n$ of $\v$.

    \begin{prop}
        Given an oriented Steiner triple system $(\S,\O(T))$ of size $n$, let $M=M(\S,\O(T))$ be the $n \times n$ matrix such that $M_{i,j}=s_i\times s_j.$
        If $\v,\w\in \R^{\S}$ then $\v\times \w=\tr([\v]^TM[\w]).$ 
    \end{prop}
    \begin{proof}
        Since the function $\tr([\v]^TM[\w])$ is bilinear (as it is a composition of bilinear functions), it is sufficient to prove the proposition for the case $\v=s_i$ and $\w=s_j$. Note that $\tr([\v]^TM[\w])=\tr([\w][\v]^TM)$ and that $[\w][\v]^T$ is an $n\times n$ matrix whose $(j,i)$-entry is $1$ and with every other entry equal to 0. Thus the $(j,j)$-entry is the only nonzero diagonal entry of $[\w][\v]^TM$; specifically, the $(j,j)$-entry of $[\w][\v]^TM$ is the $(i,j)$-entry of $M$, which by definition is $s_i\times s_j$. From this, we conclude that $s_i\times s_j=\tr([s_j]M[s_i]^T).$
By bilinearity, we have $\v\times \w=\tr([\v]^TM[\w])$ for all $\v,\w\in \R^{\S}$.
    \end{proof}
   Now we can explore the Steiner product entirely in terms of matrix-multiplication by $M$.
   \begin{prop}{\label{propAV}}
       Let $(\S,\O(T))$ be an oriented Steiner triple system of size $n$, let $\v, \w\in\R^\S$, and let $[\v]$ be the column vector $$[\v]=\begin{bmatrix}
           v_1\\
           v_2\\
           \vdots\\
           v_n
       \end{bmatrix}.$$ Then there is a unique $n\times n$ matrix $A_{\w}$ such that $A_{\w}[\v]=[[\w]^T M[\v]].$ Furthermore, $\w$ is a zero-divisor if and only if $\text{rank}(A_{\w})<n-1$.
   \end{prop}
   \begin{proof}
       Recall that $\w\times \v=\tr([\w]^TM[\v])$. Note that $[\w]^TM$ is a $1\times n$ row vector with entries in $\R^\S$. We will form the $n\times n$ matrix $A_{\w}$ whose $(i,j)$-entry is the $\R$-coefficient of $s_i$ in the Steiner product $\w\times s_j$ for all $1\leq i,j\leq n$ and note that $$A_{\w} [\v]=[[\w]^TM [\v]].$$ We will show that $A_{\w}$ is skew-symmetric. Let $c$ be the $s_i$-coefficient of the Steiner product $\w\times s_j$, and let $\{s_i,s_j,s_k\}\in T$. Then $c=f(s_k,s_j)w_k$, and the $s_j$-coefficient of $\w\times s_i$ is $f(s_k,s_j)w_k=-c$. Thus $A_{\w}$ is skew-symmetric and so $$A_{\w}^T[\v]=[[\v]^TM[\w]]=[-\w\times \v]$$ since $M$ is also skew-symmetric.

       Now suppose that $\w$ is a zero-divisor. Then there is some vector $\v$ such that $\pmb{\langle} \v\pmb{\rangle}\neq\pmb{\langle} \w\pmb{\rangle}$ and $\w\times \v=0$. Then $[\w]^TM[\v]=0$ and so $[\w]^TM[\v]=[\w]^TA_{\w}[\v]=0$ and so $A_{\w}[\v]=0$, so $[\v] \in \ker(A_{\w})$. Since we know that $\text{rank}(A_{\w})\leq n -1$ because $[\w]\in\ker(A_{\w})$, we know $\pmb{\langle} \v,\w\pmb{\rangle}\subseq\ker(A_{\w})$ and so $\text{rank}(A_{\w})<n-1$.

       Now suppose $\text{rank}(A(\v))<n-1$. Then $\dim\ker(A(\v))\geq 2$ and so there is some two-dimensional subspace of $\R^\S$ $\pmb{\langle} \v,\w\pmb{\rangle}$ such that $\pmb{\langle} \v,\w\pmb{\rangle}\subseq \ker A(\v)$. Then $[\w]^TA(\v)=0$ and so $\v\times \w=0$, so $\v$ is a zero-divisor.
   \end{proof}

\begin{defn}
    For a given $\w\in(\R^\S,\times)$, we will call the matrix $A_\w$ as constructed in Proposition \ref{propAV} the \textbf{companion matrix} of $\w$. Refer to Appendix \ref{app3} for Maple code that can construct the companion matrix of a given vector.
\end{defn}
   
   \begin{ex}
       Using the vector $\w=s_1+s_5\in\R^\S$ from Example \ref{zd}, we find that $$A_{\w}=\begin{pmatrix}
           0&0&0&1&0&0&0\\
           0&0&1&0&0&0&-1\\
           0&-1&0&0&0&-1&0\\
           -1&0&0&0&1&0&0\\
           0&0&0&-1&0&0&0\\
           0&0&1&0&0&0&-1\\
           0&1&0&0&0&1&0\\
       \end{pmatrix}.$$  $A(\v)$ has a rank of $4$, which is consistent with $\v$ being a zero-divisor.
   \end{ex}
   \begin{ex}
       If we let $\w=s_1+s_5$ and $\v=\ds\sum_{i=1}^7w_is_i\in\R^\S$ then $$[\w\times \v]=\begin{bmatrix}
           -v_4\\-v_3+v_7\\v_2+v_6\\v_1-v_5\\v_4\\-v_3+v_7\\-v_2-v_6
       \end{bmatrix}.$$ The zero-divisor locus of $\w$ is determined by the vanishing of this vector, a 3-plane in $\R^7$.
   \end{ex}
   \begin{ex}
       Let $(\S,\O_{1}(T))$ be the oriented Steiner triple system on nine elements with an automorphism group of size $27$ from Theorem \ref{nine}. Let $\w=s_1+s_2+s_3$.  Then $$A_{\w}=\left(\begin{array}{ccccccccc}
0 & -1 & 1 & 0 & 0 & 0 & 0 & 0 & 0 
\\
 1 & 0 & -1 & 0 & 0 & 0 & 0 & 0 & 0 
\\
 -1 & 1 & 0 & 0 & 0 & 0 & 0 & 0 & 0 
\\
 0 & 0 & 0 & 0 & 0 & 0 & -1 & -1 & -1 
\\
 0 & 0 & 0 & 0 & 0 & 0 & -1 & -1 & -1 
\\
 0 & 0 & 0 & 0 & 0 & 0 & -1 & -1 & -1 
\\
 0 & 0 & 0 & 1 & 1 & 1 & 0 & 0 & 0 
\\
 0 & 0 & 0 & 1 & 1 & 1 & 0 & 0 & 0 
\\
 0 & 0 & 0 & 1 & 1 & 1 & 0 & 0 & 0 
\end{array}\right),
$$ which has a rank of 4. The kernel is generated by $s_4-s_5$, $s_4-s_6$, $s_7-s_8$, $s_7-s_9$, and $\w$. Thus $\w$ is a zero-divisor which we can see by explicitly calculating $$(s_1+s_2+s_3)\times(s_4-s_5)=s_7-s_9+s_9-s_8+s_8-s_7=0.$$
   \end{ex}
   \begin{rmk}
       In 1998, Moreno proved that the space of pairs of zero-divisors of the sedenions $\mathbb{S}$ is homeomorphic to the exceptional Lie group $G_2$ \cite{moreno}. In the future it may be interesting to identify the space of zero-divisors induced by some other Steiner products. 
   \end{rmk}
   The following proposition illustrates a fundamental difference between the two basic oriented Steiner triple systems on $7$ elements.
   \begin{prop}{\label{rankgrowth}}
       Consider the two oriented Steiner triple systems $$(\S,\O(T))=\{[s_1,s_2,s_3],[s_1,s_4,s_5],[s_1,s_7,s_6],[s_2,s_4,s_6],[s_2,s_5,s_7],[s_3,s_4,s_7],[s_3,s_6,s_5]\},$$ $$(\S,\O'(T))=\{[s_1,s_2,s_3],[s_1,s_4,s_5],[s_1,s_7,s_6],[s_2,s_4,s_6],[s_2,s_5,s_7],[s_3,s_4,s_7],[s_3,s_5,s_6]\}.$$ Let $\v=\sum_{i=1}^7v_is_i$ and $\w=\sum_{i=1}^7w_is_i$ be general elements in $\R^{\S}$. $$\rm{For} \ (\S,\O(T)),\ \rm{we\ have}\  \text{rank}\begin{pmatrix}
           [\v]&[L^1_{\w}(\v)]&\cdots&[L_{\w}^6(\v)]
       \end{pmatrix}=3,$$$$\rm{and \ for} \ (\S,\O'(T)),\ \rm{we\ have}\ \text{rank}\begin{pmatrix}
           [\v]&[L^1_{\w}(\v)]&\cdots&[L_{\w}^6(\v)]
           \end{pmatrix}=7.$$
   \end{prop}
   \begin{proof}
       Using the Maple code in Appendix \ref{app3}, we have for general $\v=\sum_{i=1}^7v_is_i$ and $\w=\sum_{i=1}^7w_is_i$, that $$\text{rank}\begin{pmatrix}
           [\v]&[L^1_{\w}(\v)]&[L_{\w}^2(\v)]&[L_{\w}^3(\v)]
       \end{pmatrix}=3.$$ Therefore by Proposition \ref{plateau}, we know that $\text{rank}\begin{pmatrix}
           [\v]&[L_{\w}(\v)]&\cdots&[L_{\w}^6(\v)]
       \end{pmatrix}=3$.

       For the case of $(\S,\O'(T))$, if we let $\v=s_1+s_2,\, \w=s_2+s_3+s_4$ then we can calculate that $$\text{rank}\begin{pmatrix}
           [\v]&[L_{\w}(\v)]&\cdots&[L_{\w}^6(\v)]
       \end{pmatrix}=7.$$
   \end{proof}
   \begin{rmk}
       Consider the oriented Steiner triple system $(\S,\O'(T))$ from Proposition 7. Let $\v=s_7$. By varying $\w$, we have $\text{rank}\begin{pmatrix}
           [\v]&[L_{\w}(\v)]&\cdots&[L_{\w}^6(\v)]
       \end{pmatrix}=r$ every $r\in \{1,2,3,4,5,6,7\}$.
       
If $\w=0$ then $\text{rank}\begin{pmatrix}
           [\v]&[L^1_{\w}(\v)]&\cdots&[L_{\w}^6(\v)]
       \end{pmatrix}=1.$ 
       
       If $\w=s_3$ then $\text{rank}\begin{pmatrix}
           [\v]&[L^1_{\w}(\v)]&\cdots&[L_{\w}^6(\v)]
       \end{pmatrix}=2.$
       
       If $\w=s_1+s_3$ then $\text{rank}\begin{pmatrix}
           [\v]&[L^1_{\w}(\v)]&\cdots&[L_{\w}^6(\v)]
       \end{pmatrix}=3.$
       
       If $\w=s_1+s_2+s_3$ then $\text{rank}\begin{pmatrix}
           [\v]&[L^1_{\w}(\v)]&\cdots&[L_{\w}^6(\v)]
       \end{pmatrix}=4.$
       
       If $\w=s_1+s_2+s_3+s_6$ then $\text{rank}\begin{pmatrix}
           [\v]&[L^1_{\w}(\v)]&\cdots&[L_{\w}^6(\v)]
       \end{pmatrix}=5.$
       
       If $\w=s_1+s_2+s_3+s_4$ then $\text{rank}\begin{pmatrix}
           [\v]&[L^1_{\w}(\v)]&\cdots&[L_{\w}^6(\v)]
       \end{pmatrix}=6.$
       
       If $\w=s_1+s_2+s_3+s_7$ then $\text{rank}\begin{pmatrix}
           [\v]&[L^1_{\w}(\v)]&\cdots&[L_{\w}^6(\v)]
       \end{pmatrix}=7.$
   \end{rmk}

   \
   
   Let $\v,\w \in \R^{\S}$ be general vectors.
   Let $LN^i_{\w}(\v)=\frac{L^i_{\w}(\v)}{|L^i_{\w}(\v)|}$ denote the normalization of $L^i_{\w}(\v)$. The sequence $LN^0_{\w}(\v),LN^1_{\w}(\v),LN^2_{\w}(\v), \dots$ is a sequence of vectors on the unit sphere in $\R^{\S}$. Experimentally, one observes that $$\lim_{n\rightarrow \infty}\frac{1}{n}\sum_{i=0}^n LN^i_{\w}(\v)=0.$$ A closer look suggests that the sequence converges to a repeating pattern of the form $\mathbf{a},\mathbf{b},-\mathbf{a},-\mathbf{b}$ with $\mathbf{a},\mathbf{b},\w$ all mutually orthogonal. Let $P$ denote the plane spanned by $\mathbf{a},\mathbf{b}$. If $\w$ is fixed but we change $\v$ then, experimentally, the new repeating pattern is of the form $Q\mathbf{a},Q\mathbf{b},-Q\mathbf{a},-Q\mathbf{b}$ for some orthogonal matrix $Q$ whose restriction to $P^\perp$ is the identity. The following paragraphs explain these observations.

In Proposition~\ref{propAV}, we see that skew symmetric matrices play a fundamental role in the dynamics of the iterated Steiner product with a fixed vector $\w$. Real skew symmetric matrices are normal and have pure imaginary eigenvalues. By the spectral theorem, they can be diagonalized using a complex unitary matrix and can be block diagonalized using a real orthogonal matrix. If $A$ is a real skew symmetric $n \times n$ matrix then there exists an $n\times n$ orthogonal matrix $Q\in O(n)$ that conjugates $A$ to the following form \cite{zumino1962normal, thompson1988normal}:

$$Q\hskip 1pt A \hskip 1pt Q^T=\left(\begin{array}{cccccccccc}
0 & \lambda_1 & 0 & 0 & \dots & 0 & 0 &  &  &
\\
 -\lambda_1 & 0 & 0 & 0 & \dots & 0 & 0 &  &  &
\\
 0 & 0 & 0 & \lambda_2 & \dots & 0 & 0 &  &  &
\\
 0 & 0 & -\lambda_2 & 0 & \dots  & 0 & 0 &  &  &
\\
 \vdots & \vdots & \vdots & \vdots & \ddots & \vdots & \vdots &  &  &
\\
 0 & 0 & 0 & 0 & \dots & 0 & \lambda_k &  &  &
\\
 0 & 0 & 0 & 0 & \dots & -\lambda_k & 0 &  &  &
\\
  &  &  &  &  &  &  & 0 &   & 
  \\
  &  &  &  &  &  &  &  & \ddots  &
\\
  &  &  &  &  &  &  &  &  & 0
\end{array}\right).
$$
Without loss of generality, we can assume that $|\lambda_1|\geq |\lambda_2| \geq \dots \geq |\lambda_k |$. The non-zero eigenvalues of $A$ are $\pm i\lambda_1, \dots, \pm i\lambda_k$. If $A$ is an $n\times n$ matrix with $n$ odd then $A$ has an odd number of zero eigenvalues.

Let $\w \in \R^{\S}$. Let $A_{\w}$ be the skew symmetric matrix representing the map $L_{\w}: \R^{\S} \rightarrow \R^{\S}$ in terms of the ordered basis $s_1, \dots, s_n$. The dimension of the vector space $\pmb{\langle}\v, L^1_{\w}(\v), \dots, L^k_{\w}(\v)\pmb{\rangle}$ is the same as the dimension of the vector space $V=\pmb{\langle} [\v], A_{\w}[\v], \dots, A_{\w}^k [\v] \pmb{\rangle}$. Let $Q$ be the orthogonal matrix that block diagonalizes $A_{\w}$ as described in the previous paragraph. Let $q_p$ denote the $p^{th}$ column of $Q$. The two dimensional vector space spanned by $q_{2p-1}, q_{2p}$ is an invariant subspace for $A_{\w}$ associated to the complex eigenvector pair $\pm i\lambda_p$. If $\mathbf{u}$ is a vector in this subspace, then the result of multiplying this vector by $A_{\w}$ is a rotation by $\pi/2$ and a stretching by the factor $|\lambda_p|$. Let $\pm i\lambda^\prime_1, \dots, \pm i\lambda^\prime_r$ be the list of {\it distinct} nonzero complex eigenvalue pairs of $A_{\w}$. Associated to $\pm i \lambda^\prime_j$ is an invariant space $V_j$ consisting of the span of the 2-dimensional invariant spaces corresponding to the $\pm i \lambda_l$ pairs that are equal to $\pm i \lambda^\prime_j$. Note that if $\v\in V_j$ then the vector space $\pmb{\langle} \bar{\v}, L^1_{\w}(\bar{\v}), L^2_{\w}(\bar{\v})\pmb{\rangle}$ has dimension two. Let $N$ denote the null space of $A_{\w}$. Given $\v \in \R^{\S}$, we can express $\v$ uniquely as $\v = \v_1 + \v_2 + \dots + \v_r + \mathcal{N}$ with $\v_j \in V_j$ and $\mathcal{N} \in N$. The following theorem follows immediately from these observations.

\begin{thm}{\label{thmdyn}}
Let $\v, \w \in \R^{\S}$ with $|\S|=n$. Let $\pm i\lambda^{\prime}_1, \dots \pm i\lambda^{\prime}_r$ be the list of distinct non-zero eigenvalues of $A_{\w}$ arranged to satisfy $|\lambda^\prime_1|> |\lambda^\prime_2| > \dots > |\lambda^\prime_r |$ . Let $V_j$ denote the invariant space $\mathcal{E}_{i\lambda'_j}(A_\w)\oplus\mathcal{E}_{-i\lambda'_j}(A_\w)$ associated to $\pm i\lambda^{\prime}_j$. Let $N$ denote the null space of $A_{\w}$.
Let $\v = \v_1 + \v_2 + \dots + \v_r + \mathcal{N}$ with $\v_j \in V_j$ and $\mathcal{N} \in N$. Let $p=|\{m|\v_m\neq \mathbf{0}\}|$ (i.e. the number of nonzero $\v_j$ in the expression for $\v$).
\begin{itemize}
  \item  The dimension of the vector space $\pmb{\langle}\v, L^1_{\w}(\v), \dots, L^n_{\w}(\v)\pmb{\rangle}$ is
$2p$ if $\mathcal{N}=\mathbf{0}$.
 \item The dimension of the vector space $\pmb{\langle}\v, L^1_{\w}(\v), \dots, L^n_{\w}(\v)\pmb{\rangle}$ is
$2p+1$ if $\mathcal{N}\neq \mathbf{0}$.
\item  The dimension of the vector space $\pmb{\langle} L^1_{\w}(\v), \dots, L^n_{\w}(\v)\pmb{\rangle}$ is
$2p$.
\item The limit $\lim_{t\rightarrow \infty} \frac{L^{4t}_{\w}(\v)}{|L^{4t}_{\w}(\v)|}$ exists and is a vector $\bar{\v}\in V_j$ where $j=\min\{m | \v_m \neq \mathbf{0}\}$.
\item The dimension of the vector space $\pmb{\langle} \bar{\v}, L^1_{\w}(\bar{\v}), L^2_{\w}(\bar{\v})\pmb{\rangle}$ is two.
\item $L^m_{\w}(\bar{\v})=-L^{m+2}_{\w}(\bar{\v})$ for all $m$.
\item $\lim_{t\rightarrow \infty} \frac{1}{t}\sum_{i=1}^t L^i_{\w}(\v)=\mathbf{0}$
\end{itemize}
\end{thm}

\section{Concluding Remarks and Observations}
Hurwitz's theorem states that normed division algebras are only possible in 1, 2, 4 and 8 dimensions \cite{hurwitz1922komposition, chevalley1996algebraic, eckmann1942gruppentheoretischer, radon1922lineare}. The cross product in dimensions 3 and 7 is formed from the product operation of the normed division algebra by restricting it to the 3 and 7 imaginary dimensions of the algebra. In general, cross products exist only in three and seven dimensions \cite{WSM}. More specifically,
the 3-dimensional cross product can be expressed in terms of the quaternions $\mathbb{H}$ and the 7-dimensional cross product can be expressed in terms of the octonions $\mathbb{O}$. 
In the quaternion case, we can identify $\R^3$ with the imaginary quaternions. The cross product is then given in terms of quaternion multiplication by taking the imaginary part of the product. In the octonion case, we can identify $\R^7$ with the imaginary octonions. The cross product is then given in terms of octonion multiplication by taking the imaginary part of the product. We can see this illustrated in the following multiplication tables for the quaternions and octonions.

\begin{center}
    \begin{tabular}{c|c|c|c|c}
         $\times$&$1$&$\mathbf{i}$&$\mathbf{j}$&$\mathbf{k}$  \\
         \hline
         1&1&$\mathbf{i}$&$\mathbf{j}$&$\mathbf{k}$\\ 
         \hline
         $\mathbf{i}$&$\mathbf{i}$&$-1$&$\mathbf{k}$&$-\mathbf{j}$\\ 
         \hline
         $\mathbf{j}$&$\mathbf{j}$&$-\mathbf{k}$&$-1$&$\mathbf{i}$\\ 
         \hline
         $\mathbf{k}$&$\mathbf{k}$&$\mathbf{j}$&$-\mathbf{i}$&$-1$\\ 
    \end{tabular}
\end{center}
\begin{center}
    \begin{tabular}{c|c|c|c|c|c|c|c|c}
         $\times$&$1$&$\mathbf{e}_1$&$\mathbf{e}_2$&$\mathbf{e}_3$ &$\mathbf{e}_4$&$\mathbf{e}_5$&$\mathbf{e}_6$&$\mathbf{e}_7$  \\
         \hline
         $1$&$1$&$\mathbf{e}_1$&$\mathbf{e}_2$&$\mathbf{e}_3$ &$\mathbf{e}_4$&$\mathbf{e}_5$&$\mathbf{e}_6$&$\mathbf{e}_7$  \\
         \hline
         $\mathbf{e}_1$ &$\mathbf{e}_1$&$-1$&$\mathbf{e}_3$&$-\mathbf{e}_2$ &$\mathbf{e}_5$&$-\mathbf{e}_4$&$\mathbf{e}_7$&$-\mathbf{e}_6$  \\
         \hline
         $\mathbf{e}_2$&$\mathbf{e}_2$&$-\mathbf{e}_3$&$-1$&$\mathbf{e}_1$ &$\mathbf{e}_6$&$-\mathbf{e}_7$&$-\mathbf{e}_4$&$\mathbf{e}_5$  \\
         \hline
         $\mathbf{e}_3$ &$\mathbf{e}_3$&$\mathbf{e}_2$&$-\mathbf{e}_1$&$-1$ &$-\mathbf{e}_7$&$-\mathbf{e}_6$&$\mathbf{e}_5$&$\mathbf{e}_4$  \\
          \hline
         $\mathbf{e}_4$ &$\mathbf{e}_4$&$-\mathbf{e}_5$&$-\mathbf{e}_6$&$\mathbf{e}_7$ &$-1$&$\mathbf{e}_1$&$\mathbf{e}_2$&$-\mathbf{e}_3$  \\
          \hline
         $\mathbf{e}_5$ &$\mathbf{e}_5$&$\mathbf{e}_4$&$\mathbf{e}_7$&$\mathbf{e}_6$ &$-\mathbf{e}_1$&$-1$&$-\mathbf{e}_3$&$-\mathbf{e}_2$  \\
          \hline
          $\mathbf{e}_6$&$\mathbf{e}_6$&$-\mathbf{e}_7$&$\mathbf{e}_4$&$-\mathbf{e}_5$ &$-\mathbf{e}_2$&$\mathbf{e}_3$&$-1$&$\mathbf{e}_1$  \\
          \hline
         $\mathbf{e}_7$ &$\mathbf{e}_7$&$\mathbf{e}_6$&$-\mathbf{e}_5$&$-\mathbf{e}_4$ &$\mathbf{e}_3$&$\mathbf{e}_2$&$-\mathbf{e}_1$&$-1$  \\
    \end{tabular}
\end{center}

Observe that the binary operation on $\mathbf{a},\mathbf{b}\in \mathbf{Im}(\mathbb{H})$ given by taking $\mathbf{Im}(\mathbf{a}\times\mathbf{b})$ is equal to the Steiner product induced by the oriented Steiner triple system $(\S,\O(T))=(\{\mathbf{i},\mathbf{j},\mathbf{k}\},[\mathbf{i},\mathbf{j},\mathbf{k}])$; and the binarty operation on $\mathbf{a},\mathbf{b}\in\mathbf{Im}(\mathbb{O})$ given by taking $\mathbf{Im}(\mathbf{a}\times\mathbf{b})$ is equal to the Steiner product induced by the oriented Steiner triple system $(\S,\O_1(T))$ given in Theorem \ref{seven}.

It is worth pointing out that the octonions can be used to build several examples of Moufang loops. A Moufang loop is very much like a group except that the binary operation is only required to satisfy a weakened version of the associativity axiom. 

Steiner triple systems are a special case of Steiner systems which are in turn a special kind of block design. A Steiner system, S(t,k,n), is an n-element set S together with a distinguished set of k-element subsets of S (called blocks) with the property that each t-element subset of S is contained in exactly one block. Steiner triple systems are of the form S(2,3,n). It is an interesting endeavor to better understand and classify various types of block designs and to connect them with algebraic objects. Automorphism groups of Steiner systems (and block designs in general) can also be very interesting. For instance, the Mathieu groups are five sporadic simple groups denoted $M_{11}, M_{12}, M_{22}, M_{23}, M_{24}$. The Mathieu groups $M_{11}, M_{12}, M_{23}, M_{24}$ are the automorphism groups of Steiner systems on 11, 12, 23, and 24 points while $M_{22}$ is an index two subgroup of the automorphism group of a Steiner system on 22 points.  Mendelsohn \cite{mendelsohn1978groups} proved that any finite group G is isomorphic to the automorphism group of some Steiner triple system. Doyen and Kantor \cite{doyen2022automorphism} showed that if G is a finite group then there is an integer $M_G$ such that, for $n \geq M_G$ and $n\equiv 1 \ {\rm or}\ 3 \mod{6}$, there is a Steiner triple system $S(2,3,n)$ which has an automorphism group isomorphic to $G$. When $G$ is the trivial group, it is known that $M_G=15$.

We conclude by posing the following open questions.

\begin{qstn}
    Is is known that there are two isomorphism classes of Steiner triple systems of size thirteen \cite{HO}. What are the isomorphism classes of oriented Steiner triple systems of size thirteen?
\end{qstn}
\begin{qstn}
    What is the space of zero-divisors of $(\S,\O_3(T))$ or $(\S,\O_4(T))$ from Theorem \ref{seven}, or any of the of the oriented Steiner triple systems of size nine?
\end{qstn}
\begin{qstn}
    What is the relationship between $\Aut(\R^\S,\times)$ and $\Aut(\S,\O(T))$ in general?
\end{qstn}

\printbibliography

\section*{Appendix}
\begin{app}\label{sts237}
\normalfont
Below is Maple code used for Theorem \ref{seven}.
\vspace{\baselineskip}
\\
 restart;
\\
	with(LinearAlgebra): with(combinat): with(Statistics): with(ListTools):
\\
T := \{\{1, 2, 3\}, \{1, 4, 5\}, \{1, 6, 7\}, \{2, 4, 6\}, \{2, 5, 7\}, \{3, 4, 7\}, \{3, 5, 6\}\}:
\\
PE := permute([1, 2, 3, 4, 5, 6, 7]): 
\\
H := \{\}: 
\\
CC := \{\}: 
\\
PS := powerset(7):
\\
VV := Array([seq(numelems(T minus \{seq(map(i -$>$ PE[c][i], T[k]), k = 1 .. 7)\}), c = 1 .. 5040)]):
\\
for i to 5040 do

	if VV[i] = 0 then H := H union \{i\}: end if:
\\
end do:
\\
numelems(H);
\\
Rsort := L -$>$ SelectFirst(sort([L, Rotate(L, 1), Rotate(L, 2)])):
\\
AA := Rsort$\sim$([[1, 2, 3], [1, 4, 5], [1, 6, 7], [2, 4, 6], [2, 5, 7], [3, 4, 7], [3, 5, 6]]):
\\
BB := Rsort$\sim$([[2, 1, 3], [4, 1, 5], [6, 1, 7], [4, 2, 6], [5, 2, 7], [4, 3, 7], [5, 3, 6]]):
\\
TT := \{seq(\{seq(AA[PS[c][i]], i = 1 .. numelems(PS[c]))\} union \{seq(BB[(PS[128] minus PS[c])[j]], j = 1 .. numelems(PS[128] minus PS[c]))\}, c = 1 .. 128)\}:
\\
while 0 $<$ numelems(TT) do

 \   CC := CC union \{TT[1]\}:
    
 \   WWW := \{seq(Rsort$\sim$(\{seq(map(i -> PE[H[c]][i], TT[1][k]), k = 1 .. 7)\}), c = 1 .. 168)\}:
    
 \   TT := TT minus WWW:
    
end do
\\
numelems(TT);
\\
OR:=seq(numelems(\{seq(Rsort$\sim$(\{seq(map(i -$>$ PE[H[c]][i], CC[l][k]), k = 1 .. 7)\}), c = 1 .. 168)\}), l = 1 .. numelems(CC));
\\
GR := seq(168/OR[k], k = 1 .. numelems([OR]))
\end{app}
The function numelems(H) outputs the number of elements in the automorphism group of the Steiner system H. It has 168 elements. CC is a set of 4 sets. Each of the 4 sets in CC is a representative of the collection of oriented Steiner systems which are isomorphic to the representative under the action of the automorphism group of the Steiner system H. OR is a sequence giving the number of elements in each of these isomorphism classes. GR
is a sequence giving the size of the automorphism group of each element in CC. The size of the automorphism group of each oriented Steiner system is   168/``size of orbit".

\begin{app}\label{sts239}
\normalfont
Below is Maple code used for Theorem \ref{nine}.
\vspace{\baselineskip}
\\
restart;
\\
    with(LinearAlgebra): with(combinat): with(Statistics): with(ListTools):
\\
T := \{\{1, 2, 3\}, \{1, 4, 7\}, \{1, 5, 9\}, \{1, 6, 8\}, \{2, 4, 9\}, \{2, 5, 8\}, \{2, 6, 7\}, \{3, 4, 8\}, \{3, 5, 7\}, \{3, 6, 9\}, \{4, 5, 6\}, \{7, 8, 9\}\}:
\\
PE := permute([1, 2, 3, 4, 5, 6, 7, 8, 9]):
\\
H := \{\}:
\\
CC := \{\}:
\\
PS := powerset(12):
\\
VV := Array([seq(numelems(T minus \{seq(map(i -$>$ PE[c][i], T[k]), k = 1 .. 12)\}), c = 1 .. 362880)]):
\\
for i to 362880 do

	if VV[i] = 0 then H := H union \{i\}: end if:
\\
end do:
\\
numelems(H):
\\
Rsort := L -$>$ SelectFirst(sort([L, Rotate(L, 1), Rotate(L, 2)])):
\\
AA := Rsort$\sim$([[1, 2, 3], [4, 5, 6], [7, 8, 9], [1, 4, 7], [2, 5, 8], [3, 6, 9], [1, 5, 9], [2, 6, 7], [3, 4, 8], [1, 6, 8], [2, 4, 9], [3, 5, 7]]):
\\
BB := Rsort$\sim$([[2, 1, 3], [5, 4, 6], [8, 7, 9], [4, 1, 7], [5, 2, 8], [6, 3, 9], [5, 1, 9], [6, 2, 7], [4, 3, 8], [6, 1, 8], [4, 2, 9], [5, 3, 7]]):
\\
TT := \{seq(\{seq(AA[PS[c][i]], i = 1 .. numelems(PS[c]))\} union \{seq(BB[(PS[4096] minus PS[c])[j]], j = 1 .. numelems(PS[4096] minus PS[c]))\}, c = 1 .. 4096)\}:
\\
while 0 $<$ numelems(TT) do

	CC := CC union \{TT[1]\}:
    
	WWW := \{seq(Rsort$\sim$(\{seq(map(i -$>$ PE[H[c]][i], TT[1][k]), k = 1 .. 12)\}), c = 1 .. 432)\}:
    
	TT := TT minus WWW:
\\
end do:
\\
numelems(TT):
\\
OR:=seq(numelems(\{seq(Rsort$\sim$(\{seq(map(i -$>$ PE[H[c]][i], CC[l][k]), k = 1 .. 12)\}), c = 1 .. 432)\}), l = 1 .. 16);
\\
GR := seq(432/OR[k], k = 1 .. numelems([OR]))
\end{app}
The function numelems(H) outputs the number of elements in the automorphism group of the Steiner system H. It has 432 elements.
The very last line gives the size of the orbits. There are 16 orbits, each representing one isomorphism class. The size of the automorphism group of each oriented Steiner system is   432/``size of orbit". The set CC stores a representative of each isomorphism class.

\begin{app}{\label{app3}}
\normalfont
    Below is Maple code used for Proposition \ref{rankgrowth}. Change the entries in $W$ to obtain different matrices $AW=A_{\w}$. It is easy to check that AW.V is equal to cro(W,V).
\vspace{\baselineskip}
\\
restart;
\\
with(LinearAlgebra):
\\
T := [[1, 2, 3], [1, 4, 5], [1, 7, 6], [2, 4, 6], [2, 5, 7], [3, 4, 7], [3, 6, 5]]:
\\
A := Matrix(7, 7):
\\
for i to 7 do

    A[T[i][1], T[i][2]] := s[T[i][3]]:
    
    A[T[i][2], T[i][3]] := s[T[i][1]]:
    
    A[T[i][3], T[i][1]] := s[T[i][2]]:
    \\
end do:
\\
M := A - Transpose(A):
\\
cro := (v, u) -$>$ (Matrix([[seq(coeff(Trace(((v$^{\%T}$) . M) . u), x[i]), i = 1 .. 7)]]))$^{\%T}$
\\
V := Transpose(Matrix([v[1], v[2], v[3], v[4], v[5], v[6], v[7]])):
\\
W := Transpose(Matrix([1,2,-3,1,1,0,2])):
\\
R := Matrix([V, cro(W, V), cro(W,cro(W, V))]):
\\
Rank(R)
\\
AA := cro(W, V)
\\
AW := Matrix([seq([seq(coeff(AA[k, 1], v[i]), i = 1 .. 7)], k = 1 .. 7)])

\end{app}


\end{document}